\documentclass[12pt]{amsart}

\usepackage{amsmath}
\usepackage{amsthm}
\usepackage{amssymb}
\usepackage{mathrsfs}
\usepackage{ifthen}
\usepackage{graphicx}
\usepackage[T1]{fontenc} 
\usepackage{enumerate}
\usepackage{hyperref}
\usepackage{float}
\usepackage{tikz}
\usepackage{enumitem}
\usepackage{fullpage}
\usepackage{enumerate}
\usepackage{enumitem}
\usepackage{xcolor}
\usepackage{todonotes}

 \nonstopmode \numberwithin{equation}{section}

\theoremstyle{plain}

\newtheorem{Thm}{Theorem}[section]

\newtheorem{lemma}[Thm]{Lemma}
\newtheorem{prop}[Thm]{Proposition}

\newtheorem{rem}[Thm]{Remark}

\theoremstyle{definition}

\newtheorem{defn}[Thm]{Definition}

\newtheorem{matrixcondition}[Thm]{Condition}

\newcommand{\Tr}{\operatorname{Tr}}
\newcommand{\Var}{\operatorname{Var}}
\newcommand{\Cov}{\operatorname{Cov}}

\allowdisplaybreaks

\title{CLT for LES of real valued random centrosymmetric matrices}
\author{Indrajit Jana}
\address{Indian Institute of Technology, Bhubaneswar}
\email{ijana@iitbbs.ac.in}
\thanks{Indrajit Jana - \textit{Email: ijana@iitbbs.ac.in;}\\ 
}

\author{Sunita Rani}
\address{Indian Institute of Technology, Bhubaneswar}
\email{s21ma09007@iitbbs.ac.in}
\thanks{Sunita Rani - \textit{Email: s21ma09007@iitbbs.ac.in;} (Corresponding author).\\
}

\begin{document}
\date{}

\begin{abstract}We study the fluctuations of the eigenvalues of real valued large centrosymmetric random matrices via its linear eigenvalue statistic. This is essentially a central limit theorem (CLT) for sums of dependent random variables. The dependence among them leads to behavior that differs from the classical CLT. The main contribution of this article is finding the expression of the variance of the limiting Gaussian distribution. The crux of the proof lies in combinatorial arguments that involve counting overlapping loops in complete undirected weighted graphs with growing degrees.

keywords: Centrosymmetric matrix, Linear eigenvalue statistics, Central limit theorem.

MSC2020 Subject Classification: 15B52; 60BXX; 60FXX
\end{abstract}

\maketitle


\section{Introduction} 

Random Matrix Theory (RMT) has been emerged as an important mathematical tool in various fields of science such as Mathematical Physics, Electrical Engineering, Statistics etc. \cite{BardenetHardy2020,BourgadeErdosYau2014,4493413,serfaty2024lecturescoulombrieszgases}.
We consider a random matrix $M$ of order $n \times n$. 
Let $\lambda_1, \lambda_2, \dots, \lambda_n\in \mathbb{C}$ denote its eigenvalues. 
The empirical spectral distribution is defined as  
$$
    \mu_{n}(\cdot) = \frac{1}{n}\sum_{i=1}^{n}\delta_{\lambda_{i}}(\cdot),
$$  
where $\delta_{\lambda_i}$ is the Dirac measure at $\lambda_i$. Note that $\mu_{n}$ is a random probability measure on $\mathbb{C}$, which is equipped with the Borel sigma algebra. The measure $\mu_{n}$ is a probability measure because $\mu_{n}(\mathbb{C})=1$, and $\mu_{n}$ is a random measure because it is dependent on the eigenvalues $\lambda_{1},\ldots,\lambda_{n},$ which are random variables.

Furthermore, we define the linear eigenvalue statistics (LES) as  
$$
    L_n(f) = \sum_{i=1}^n f(\lambda_i),
$$  
for a suitable test function $f:\mathbb{C}\to\mathbb{C}$. 
We are interested in the limiting behavior of the centered and rescaled 
LES. More precisely, we would like to investigate whether there exist deterministic sequences $\{c_{n}\}$ and $\{d_{n}\}$ such that $c_{n}\big(L_{n}(f) - d_{n}\big)$ converges to a non-trivial random variable. This study may be regarded as the analogue of proving a Central Limit Theorem (CLT) in classical probability.

The CLT for LES is an important topic in RMT, which is applicable statistics. In the classical CLT we study sums of independent random variables. In contrast, here we have a sum of dependent random variable; infact highly correlated random variables. This makes the problem both more challenging and more interesting. CLTs for LES of different types of matrices have been studied previously with significant contributions including \cite{Jonsson1982SomeLT,johansson1998fluctuations,sinai1998central,bai2008clt,lytova2009central}. In particular, such results are available for matrices including band and sparse matrices \cite{anderson2006clt,li2013central,shcherbina2015fluctuations,jana2016fluctuations}, Toeplitz and band Toeplitz matrices \cite{chatterjee2009fluctuations,liu2012fluctuations}, circulant matrices \cite{adhikari2017fluctuations,adhikari2018universality,maurya2021fluctuations}, and non-Hermitian matrices \cite{rider2006gaussian, o2016central, jana2022clt,kopel2015linear,cipolloni2021fluctuation, cipolloni2022fluctuations, cipolloni2023central}. In this article, we focus on the case of real valued centrosymmetric matrices as defined in Definition \ref{centrodefn}

The proof of such results requires tools from both probability theory and combinatorics. Unlike the classical CLT, the proof of convergence and the calculation of limiting variance was done using two completely different techniques. The convergence is established using probabilistic methods, while the limiting variance is computed using a different, combinatorial approach. The main contribution of this work lies in the combinatorial ideas used for the variance calculation, which is illustrated in Section \ref{sec:variance calculation}.



\section{Main result}\label{sec:main result}
We first define the centrosymmetric matrix. 
\begin{defn}[Centrosymmetric  Matrix]\label{centrodefn} A matrix is called centrosymmetric if its entries are  symmetric with respect to the center of the matrix. More precisely, let  $M = [m_{i,j}]_{n \times n}$ be a random matrix. The entries $m_{ij}$ are i.i.d. random variables, but they must satisfy the condition  
$$
m_{i, j}=m_{n-i+1, n-j+1} \text { for } i, j \in\{1, \ldots, n\} .
$$
Then $M$ is called a random centrosymmetric matrix.
    \end{defn}
\begin{figure}[h]
\centering
\begin{tikzpicture}[scale=0.8]
    \draw (0,0) rectangle (5,5);
    \foreach \y in {1,2,3,4}{
        \draw (0,\y) -- (5,\y);
    }
    \foreach \x in {1,2,3,4}{
        \draw (\x,0) -- (\x,5);
    }
    \fill[cyan!30] (5,5) rectangle (4,4);
    \fill[pink!30] (4,5) rectangle (3,4);
    \fill[teal!30] (3,5) rectangle (2,4);
    \fill[violet!30] (2,5) rectangle (1,4);
    \fill[orange!30] (1,5) rectangle (0,4);
    
    \fill[cyan!30] (0,0) rectangle (1,1);
    \fill[pink!30] (1,0) rectangle (2,1);
    \fill[teal!30] (2,0) rectangle (3,1);
    \fill[violet!30] (3,0) rectangle (4,1);
    \fill[orange!30] (4,0) rectangle (5,1);
     \fill[olive!30] (5,4) rectangle (4,3);
    \fill[magenta!30] (4,4) rectangle (3,3);
    \fill[green!30] (3,4) rectangle (2,3);
    \fill[yellow!30] (2,4) rectangle (1,3);
    \fill[lime!30] (1,4) rectangle (0,3);
     \fill[olive!30] (0,1) rectangle (1,2);
    \fill[magenta!30] (1,1) rectangle (2,2);
    \fill[green!30] (2,1) rectangle (3,2);
    \fill[yellow!30] (3,1) rectangle (4,2);
    \fill[lime!30] (4,1) rectangle (5,2);
    \fill[red!30] (0,2) rectangle (1,3);
    \fill[blue!30] (1,2) rectangle (2,3);
    \fill[gray!30] (2,2) rectangle (3,3);
    \fill[blue!30] (3,2) rectangle (4,3);
    \fill[red!30] (4,2) rectangle (5,3);
    
    \draw (0,0) grid (5,5);
    
    \draw[<->] (1.5,1.5) -- (3.5,3.5);
    \draw[dotted, <->] (0.5,0.5) -- (4.5,4.5);
    \draw[<->] (3.5,1.5) -- (1.5,3.5);
    \draw[dotted, <->] (4.5,0.5) -- (0.5,4.5);
    \draw[<->] (1.5,2.5) -- (3.5,2.5);
    \draw[dotted, <->] (0.5,2.5) -- (4.5,2.5);
    \draw[<->] (2.5,1.5) -- (2.5,3.5);
    \draw[dotted, <->] (2.5,0.5) -- (2.5,4.5);
    \draw[dotted, <->] (1.5,0.5) -- (3.5,4.5);
    \draw[dotted, <->] (3.5,0.5) -- (1.5,4.5);
    \draw[dotted, <->] (0.5,1.5) -- (4.5,3.5);
    \draw[dotted, <->] (4.5,1.5) -- (0.5,3.5);
\end{tikzpicture}
\caption{Symmetry pattern of a centrosymmetric 5$\times$5 matrix.}
\label{fig:cento_fig}
\end{figure}
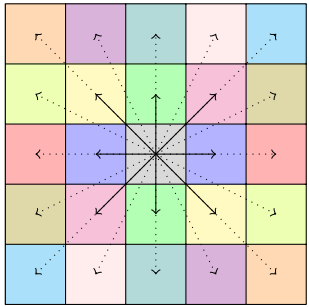

\begin{matrixcondition}\label{cond:matrixcond_real}
Let $M=[m_{ij}]_{n\times n}$ be a centrosymmetric matrix with entries 
$$
    m_{ij} = \frac{1}{\sqrt{n}}x_{ij},
$$ 
where $\{x_{ij}: 1\leq i,j\leq n\}$ are i.i.d. real valued random variables with constraint $x_{ij} = x_{n+1-i,\, n+1-j}$. We assume that $\mathbb{E}[x_{ij}] = 0$ and $\mathbb{E}[x_{ij}^2] = 1$ for all $1\leq i,j\leq n$. Moreover, the $x_{ij}$ are continuous random variables with bounded density. In addition, for every $p \in \mathbb{N}$, there exists a constant $C_p > 0$ such that  
$\mathbb{E}[x_{ij}^p] \leq C_p.$
\end{matrixcondition}

We now define the centered LES as follows.
\begin{align*}
    \operatorname{L}_{n}^{\circ}(f)=\sum_{i=1}^{n}f(\lambda_{i})-\sum_{i=1}^{n}\mathbb{E}[f(\lambda_{i})].
\end{align*}
From \cite[Theorem 5.2]{jana2024spectrum}, centrosymmetric matrix $M$ follows circular law.
Our main theorem, which is stated below, gives limiting distribution of $\operatorname{L}_{n}^{\circ}(f)$ as the matrix size increases. In what follows, the notation $\mathbb{D}=\{z\in \mathbb{C}:|z|\leq 1\}$ denotes the unit circle centered at origin, and $\partial \mathbb{D}=\{z\in \mathbb{C}:|z|=1\}$ is the boundary of $\mathbb{D}$.

\begin{Thm}\label{thm: main theorem}
Let $M$ be a random matrix satisfying Condition \ref{cond:matrixcond_real}. Let $f$ be a complex analytic function. Then $\operatorname{L}_{n}^{\circ}(f)\stackrel{d}{\to}N(0, V_{f})$, where $\stackrel{d}{\to}$ stands for convergence in distribution, and
\begin{align*}
   &V_f\\
   =&-\frac{1}{4\pi^2} 
\oint_{\partial \mathbb{D}} \oint_{\partial \mathbb{D}} 
f(z)\, f(\bar{\eta}) \Bigg(
2\left(1-z\bar\eta\right)^{-2}+\frac{4}{z\bar\eta((z\bar\eta)^{2}-1)}+4\left( \frac{1}{z\bar\eta\,(z-1)(\bar\eta-1)} \;-\; \frac{1}{z\bar\eta\,(z\bar\eta-1)} \right)\Bigg) 
\, dz\, d\bar{\eta}.
\end{align*}
The above formula boils down to $V_{f}=\sum_{k=1}^{d}2k|a_{k}|^2$ when $f(z)=\sum_{k=0}^{d}a_{k}z^{k}$ is a polynomial.
\end{Thm}

The proof of the above theorem is split into two parts. In the first part, we argue that the limit of $\operatorname{L}_{n}^{\circ}(f)$ is a Gaussian distribution. Then we find the expression of $V_{f}$ separately. Below, we outline the proof of the first part. It is essentially a consequence of \cite[Theorem 3.2]{jana2025cltles}, which proves the CLT of LES of a block diagonal matrix, where the blocks are correlated. Our strategy is to convert the centrosymmetric matrix into a block diagonal format, then apply \cite[Theorem 3.2]{jana2025cltles}.
\begin{proof}
   Assuming the centrosymmetric matrix $M$ is of even order, we may write it in block form as
   $$M  = \left(\begin{array}{cc}
		A & B\\
		C & D
	\end{array}\right)_{n\times n},$$
where $A, B, C, D$ are $n/2\times n/2$ sized matrices. Then by Weaver's theorem \cite[Theorem 9]{weaver1985centrosymmetric}\label{orthogonal}, the following matrix
$$\mathcal{M}=\left(\begin{array}{cc}
	A+J C  & 0 \\
	0 & A-J C
    \end{array}\right)$$
is orthogonally similar to the original matrix $M$. Here $J$ denotes the counter identity matrix, i.e., the square matrix with ones on the counter-diagonal and zeros elsewhere. Notice that the matrices $A+JC$ or $A- JC$ are matrices of i.i.d. entries individually. However, the $(i,j)$th entry of $A+JC$ and the corresponding $(i,j)$th entry of $A-JC$ are dependent, though uncorrelated. Therefore, by applying \cite[Theorem 3.2]{jana2025cltles} under the assumptions that the correlation coefficients are zero and the test function is analytic over $\mathbb{D}$, we conclude the proof.

If the size of the centrosymmetric matrix $M$ is odd, then also by the same Weaver's theorem, we can make $M$ orthogonally similar to a block diagonal matrix with one extra row and column added to the first block $A+JC.$ The rest of the analysis is similar to above.
\end{proof}
We can verify the result through numerical simulation; see Figure \ref{fig:CLT_thm}.
\begin{figure}[H]\label{5}
    \centering
    \includegraphics[scale=0.7]{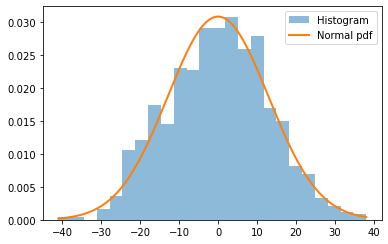}
    \caption{Histogram of $\operatorname{L}_{n}^{\circ}(f)$ for a $4000\times 4000$ random centrosymmetric matrix, averaged over 750 samples, where matrix entries are Gaussian random variables and test function is $f(x)=x^2+4x^5$. The sample variance of $\operatorname{L}_{n}^{\circ}(f)$ is 166.7225, which is approximately same as suggested by the Proposition \ref{Prop:Var}.}
    \label{fig:CLT_thm}
\end{figure}

We now shift our focus to find the expression of the limiting variance $V_{f}.$ Indeed, the expression of the limiting variance $V_f$ can also be derived from \cite[Theorem 3.2]{jana2025cltles}. However, there is an interesting hidden combinatorial structure associated to the expression of $V_{f},$ which was not analyzed in \cite[Theorem 3.2]{jana2025cltles}. Here we explore the combinatorial structure to derive the expression of $V_{f}.$


\section{Calculation of variance}\label{sec:variance calculation}
The main contribution of this paper is the following proposition, which gives the final expression of the limiting variance.
\begin{prop}\label{Prop:Var} 
Let $\operatorname{L}_{n}^{\circ}(f)$ be same as in Theorem \ref{thm: main theorem}, and $V_{f}$ be its limiting variance. Then 
\begin{align*}
   & V_{f}\\ =& -\frac{1}{4\pi^2} 
\oint_{\partial \mathbb{D}} \oint_{\partial \mathbb{D}} 
f(z)\, f(\bar{\eta}) \Bigg(
2\left(1-z\bar\eta\right)^{-2}+\frac{4}{z\bar\eta((z\bar\eta)^{2}-1)}+4\left( \frac{1}{z\bar\eta\,(z-1)(\bar\eta-1)} \;-\; \frac{1}{z\bar\eta\,(z\bar\eta-1)} \right)\Bigg) 
\, dz\, d\bar{\eta},
\end{align*}
when $f:\mathbb{C}\to\mathbb{C}$ is a complex analytic function. In particular, 
\begin{align*}
    V_{f}=\sum_{k=1}^{d} 2k|a_{k}|^2,
\end{align*}
when $f(z)=\sum_{k=1}^{d} a_{k} z^{k}$ is a polynomial of degree $d \geq 1$.
\end{prop}

To prove Proposition \ref{Prop:Var}, we need to evaluate the terms $\mathbb{E}[\Tr(M^{k})]$ and $\mathbb{E}[\Tr(M^{k})\Tr(M^{l})]$. We use combinatorial arguments to evaluate these terms.

\subsection{Analyzing expected product}

We evaluate the following terms 
$$
\mathbb{E}[\mathrm{Tr}(M^{k})] = \frac{1}{n^{k/2}}\sum_{i_{1},i_{2},\dots,i_{k}=1}^{n}\mathbb{E}[x_{i_{1}i_{2}}x_{i_{2}i_{3}} \dots x_{i_{(k-1)}i_{k}}x_{i_{k}i_{1}}],
$$
and
$$\mathbb{E}[\Tr(M^{k})\Tr(M^{l})]=\frac{1}{n^{(k+l)/2}} \sum_{{i_{1},i_{2},\dots i_{k} , j_{1}, j_{2}, \dots j_{l}=1}}^{n} \mathbb{E}\left[(x_{i_{1}i_{2}}x_{i_{2}i_{3}}\dots x_{i_{k}i_{1}})( x_{j_{1}j_{2}} x_{j_{2}j_{3}}\dots  x_{j_{l}j_{1}})\right],$$

where $x_{ij}s$  are random variables from condition \eqref{cond:matrixcond_real}. The random variables $x_{i_{1}i_{2}}, x_{i_{2}i_{3}}, \dots, x_{i_{k}i_{1}}$ are connected through their indices $i_{1}, i_{2}, \ldots, i_{k}$, which play an important role in the forthcoming analysis. For this, we introduce the notion of an \textit{index chain} as follows.

\begin{defn}[Index chains]
The unique sequence $i_{1}\to i_{2}\to\cdots\to i_{k}\to i_{1}$ of indices in the product $x_{i_{1}i_{2}}$ $x_{i_{2}i_{3}}$ $\cdots$ $ x_{i_{k}i_{1}}$ is called index chain.
    \begin{enumerate}[label=(\alph*)]
        \item Single chain: In the computation of 
          $$\mathbb{E}[\Tr(M^{k})]= \frac{1}{n^{k/2}}\sum_{i_{1},i_{2},\dots,i_{k}=1}^{n}\mathbb{E}[x_{i_{1}i_{2}}x_{i_{2}i_{3}} \dots x_{i_{(k-1)}i_{k}}x_{i_{k}i_{1}}],$$ 
          the product $(x_{i_{1}i_{2}}x_{i_{2}i_{3}} \dots x_{i_{(k-1)}i_{k}}x_{i_{k}i_{1}})$ is referred as a single chain.
        \item Double chain: In the computation of 
          $$\mathbb{E}[\Tr(M^{k})\Tr(M^{l})]=\frac{1}{n^{(k+l)/2}} \sum_{{i_{1},i_{2},\dots i_{k} , j_{1}, j_{2}, \dots j_{l}=1}}^{n} \mathbb{E}\left[(x_{i_{1}i_{2}}x_{i_{2}i_{3}}\dots x_{i_{k}i_{1}})( x_{j_{1}j_{2}} x_{j_{2}j_{3}}\dots  x_{j_{l}j_{1}})\right],$$ the product
         $(x_{i_{1}i_{2}}x_{i_{2}i_{3}}\dots x_{i_{k}i_{1}})( x_{j_{1}j_{2}} x_{j_{2}j_{3}}\dots  x_{j_{l}j_{1}})$ is refereed as double chain.
    \end{enumerate}
\end{defn}
Indices in an index chain take values from the set ${1,2,\ldots,n}.$ But while computing $\mathbb{E}[\Tr(M^{k})]$ and $\mathbb{E}[\Tr(M^{k})\Tr(M^{l})]$, not all indices are independent; some are fixed once others are chosen. We call the independent ones \textit{free indices}, and their count the \textit{degrees of freedom}.

We now evaluate $\mathbb{E}[\Tr(M^{k})]$ and $\mathbb{E}[\Tr(M^{k})\Tr(M^{l})].$ We begin with the single chain expectation in Subsection \ref{sec:Real_Single_Chain_Expectation}, and then proceed to the double chain expectation in Subsection \ref{sec:Real_Double_Chain_Expectation}.

\subsection{Single Chain Expectation}\label{sec:Real_Single_Chain_Expectation} Here we compute the expectation of the single chain, which is expressed as follows
$$
\mathbb{E}[\mathrm{Tr}(M^{k})] = \frac{1}{n^{k/2}}\sum_{i_{1},i_{2},\dots,i_{k}=1}^{n}\mathbb{E}[x_{i_{1}i_{2}}x_{i_{2}i_{3}} \dots x_{i_{(k-1)}i_{k}}x_{i_{k}i_{1}}].
$$

In the above sum product of random variables, it is possible that some of the random variables in a product are equal to each other. Let us consider two extreme cases; all are equal, and all are different. In the second case, the expectation is zero due to the fact that $\mathbb{E}[x_{ij}]=0$ by Condition \ref{cond:matrixcond_real}. In the first case, the asymptotic expectation  contribution is $n^{-k/2}\sum_{i=1}^{n}\mathbb{E}[x_{ii}^{k}]=O(n^{-k/2+1}),$ due the bounded moment assumption from Condition \ref{cond:matrixcond_real}. This term is asymptotically zero if $k>2.$  We observe that making more random variables equal reduces the degrees of freedom in the index chain. Therefore, to obtain a larger and more significant contribution, we make fewer random variables equal preferably two.

Now, before proceeding to evaluate the expectation of a single chain, we define certain types of mergings that will be needed in both single and double chain expectations, as follows. If the random variables of a chain pair up with the random variables from the same chain, we call this an \textit{intra chain merging}. If all the random variables find their pairs in the other chain, we call this a \textit{cross chain merging}. Finally, if some random variables find their pairs within the same chain while others pair with variables in the other chain, we call this a \textit{partial chain merging}. These types of mergings will be used throughout this Section \ref{sec:variance calculation}.

We now proceed to evaluate $\mathbb{E}[x_{i_{1}i_{2}}x_{i_{2}i_{3}} \dots x_{i_{(k-1)}i_{k}}x_{i_{k}i_{1}}]$ using intra chain merging. There are two possible types of intra chain merging; sequential intra chain merging and non-sequential intra chain merging, which are explained below. We use a similar argument as in the complex valued double chain calculation in \cite[Section 7.3]{jana2024spectrum}.
The only difference is that, in the double chain expectation, there is a sequential cross chain merging of two distinct chains. For example,  
$$
(x_{i_{1}i_{2}}x_{i_{2}i_{3}} \dots x_{i_{(k-1)}i_{k}}x_{i_{k}i_{1}})
(x_{j_{1}j_{2}}x_{j_{2}j_{3}} \dots x_{j_{(l-1)}i_{l}}x_{i_{l}i_{1}}).
$$  
Here, however, we have a single chain. We will treat the first half of the chain as one chain and the second half as another, for example,  
$$
(x_{i_{1}i_{2}}x_{i_{2}i_{3}} \dots x_{i_{k/2}i_{(k/2)+1}})
(x_{i_{(k/2)+1}i_{(k/2)+2}} \dots x_{i_{(k-1)}i_{k}}x_{i_{k}i_{1}}).
$$  
$$
\mathbb{E}[\Tr(M^{10})] = \frac{1}{n^{5}}\displaystyle \sum_{i_{1},i_{2},\dots,i_{10}=1}^{n}\mathbb{E}[x_{i_{1}i_{2}}x_{i_{2}i_{3}} \dots x_{i_{9}i_{10}}x_{i_{10}i_{1}}].
$$

We now discuss in detail both sequential and non-sequential intra chain mergings as follows.

\textbf{Case 1 (Sequential intra chain merging):} In sequential intra chain merging of the chain 
$$
\left(x_{i_{1}i_{2}}x_{i_{2}i_{3}}x_{i_{3}i_{4}}x_{i_{4}i_{5}}x_{i_{5}i_{6}}\right)\left(x_{i_{6}i_{7}}x_{i_{7}i_{8}}x_{i_{8}i_{9}}x_{i_{9}i_{10}}x_{i_{10}i_{1}}\right),
$$
if $x_{i_{t}i_{t+1}}$, an element from the first half of the chain, gets paired with $x_{i_{p}i_{p+1}}$, an element from the second half of the chain, then $x_{i_{t+l}i_{t+1+l}}$ will be paired with $x_{i_{p+l}i_{p+1+l}}$ for all possible $t$ and $l$, see Figure \ref{fig:k=10,Sequential intra chain merging}. We prove that this type of merging gives a nonzero asymptotic contribution.

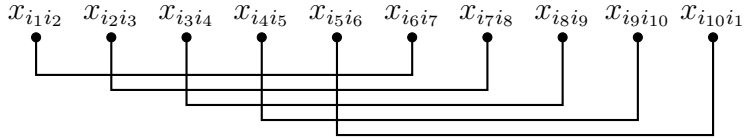
\begin{figure}[h]
    \begin{center}
\begin{tikzpicture}

   \def\x{0}
   \def\y{0}
  \filldraw (\x, \y) circle (0.06);
   \node[anchor = south] at (\x, \y) {$x_{i_{1}i_{2}}$};
  \filldraw (\x+1, \y) circle (0.06);
   \node[anchor = south] at (\x+1, \y) {$x_{i_{2}i_{3}}$};
  \filldraw (\x+2, \y) circle (0.06);
   \node[anchor = south] at (\x+2, \y) {$x_{i_{3}i_{4}}$};
  \filldraw (\x+3, \y) circle (0.06);
   \node[anchor = south] at (\x+3, \y) {$x_{i_{4}i_{5}}$};
  \filldraw (\x+4, \y) circle (0.06);
   \node[anchor = south] at (\x+4, \y) {$x_{i_{5}i_{6}}$};
  \filldraw (\x+5, \y) circle (0.06);
   \node[anchor = south] at (\x+5, \y) {$x_{i_{6}i_{7}}$};
  \filldraw (\x+6, \y) circle (0.06);
   \node[anchor = south] at (\x+6, \y) {$x_{i_{7}i_{8}}$};
  \filldraw (\x+7, \y) circle (0.06);
   \node[anchor = south] at (\x+7, \y) {$x_{i_{8}i_{9}}$};
  \filldraw (\x+8, \y) circle (0.06);
  \node[anchor = south] at (\x+8, \y) {$x_{i_{9}i_{10}}$};
   \filldraw (\x+9, \y) circle (0.06);
  \node[anchor = south] at (\x+9, \y) {$x_{i_{10}i_{1}}$};
 \draw[thick] (\x, \y) -- (\x, \y-0.5) -- (5+\x, \y-0.5) -- (\x+5, \y); 
 \draw[thick] (\x+1, \y) -- (\x+1, \y-0.7) -- (6+\x, \y-0.7) -- (\x+6, \y);
 \draw[thick] (\x+2, \y) -- (\x+2, \y-0.9) -- (7+\x, \y-0.9) -- (\x+7, \y);
  \draw[thick] (\x+3, \y) -- (\x+3, \y-1.1) -- (8+\x, \y-1.1) -- (\x+8, \y);
  \draw[thick] (\x+4, \y) -- (\x+4, \y-1.3) -- (9+\x, \y-1.3) -- (\x+9, \y);
  \end{tikzpicture}
  \end{center}
  \caption{Sequential intra chain merging}
    \label{fig:k=10,Sequential intra chain merging}
\end{figure}

Now, by using the similar argument as in \cite[Case 2(b) in Section 7.3]{jana2024spectrum}, intra chain merging for the first and second halves of the chain,
$i_1 \to i_2 \to i_3 \to i_4 \to i_5 \to i_6 \quad \text{and} \quad i_6 \to i_7 \to i_8 \to i_9 \to i_{10} \to i_1,$
can be done in two ways. One is by copying the first half of the chain over the second as follows;
\begin{align*}
    i_1 \to i_2 \to i_3 \to i_4 \to i_5 \to i_1 \to i_2 \to i_3 \to i_4 \to i_5 \to i_1,
\end{align*}
or
\begin{align*}
    (n+1-i_{1})\rightarrow (n+1-i_{2})\rightarrow \dots \rightarrow (n+1-i_{5})\rightarrow (n+1-i_{1})
\end{align*}
The second possibility is due to the fact that $x_{(i,j)}=x_{(n+1-i,n+1-j)}.$ These are possibilities with $5$ free indices. Combining these, we obtain the total contribution of intra chain merging occurred by the type of merging depicted by the Figure \ref{fig:k=10,Sequential intra chain merging} is $2\times n^{5}.$

\textbf{Case 2 (Non-sequential intra chain merging):}  
By following a similar argument as in \cite[Case 2(b) in Section 7.3]{jana2024spectrum} , it can be shown that if we do not follow sequential intra chain merging and instead try to merge the chain in a non-sequential manner, that is, if the first half of the chain does not merge with the second half in the same order, then it reduces the degrees of freedom. Therefore, this case gives an asymptotically zero contribution.

Consequently,
$$
\lim_{{n \to \infty}} \mathbb{E}[\Tr(M^{10})]
 =\lim_{{n \to \infty}} \frac{2 \times n^5+o(n^5)}{n^5}=2.
$$
Thus, we can conclude from the above cases that for the intra chain merging of the chain  $i_{1} \rightarrow i_{2} \rightarrow \dots \rightarrow i_{k} \rightarrow i_{1}$, the only possibilities with asymptotically non zero contributions are
$$
i_{1}\rightarrow i_{2}\rightarrow \dots \rightarrow i_{k/2} \rightarrow n+1-i_{1}\rightarrow n+1-i_{2} \rightarrow  \dots n+1-i_{k/2}\rightarrow i_{1}
$$
and
$$
i_{1}\rightarrow i_{2}\rightarrow \dots \rightarrow i_{k/2} \rightarrow  i_{1}\rightarrow i_{2}\rightarrow \dots \rightarrow i_{k/2} \rightarrow i_{1}.
$$
All other, except these two, give zero asymptotic contribution due to reduction in degree of freedom. The above observation can be generalized into the following lemma.
\begin{lemma}{\label{lemma: one_chain_real}}
Let $M$ be the matrix as in Theorem \ref{thm: main theorem}. Then we have
    \begin{align}
    \mathbb{E}[\Tr(M^k)]
=\left\{\begin{array}{ll}
 2  & \text{ if } k= 2,4,\dots \\
0  & \text{ if } k= 1,3,\dots
\end{array}\right. + o(1) .\notag
\end{align}
\end{lemma}
\begin{proof}
There are two distinct scenarios: one where $k$ is odd, and the other where $k$ is even. We provide the proofs separately for each case.

\textbf{Case 1: When $k$ is odd:} 
As we observed at the beginning of Section \ref{sec:Real_Single_Chain_Expectation}, making all random variables unequal gives zero contribution. The significant contribution comes from the case when there is no random variables with single power (because $\mathbb{E}[x_{ij}]=0$) i.e., the random variables must appear with at least power $2.$ If all the the random variables are having power exactly equal to $2,$ then the degree of freedom in the index chain is $k/2.$ But if there are total odd many random variables, i.e., $k$ odd, there will be at least one random variable having power at least $3,$ and all other random variables having power at least $2.$ This will make the degree of freedom of the index chain strictly less than $k/2$. As a result, the asymptotic contribution is is $o(1).$

\textbf{Case 2: When $k$ is even:} In the scenario where the matrix size is even, that is, $k = 2r$, we have

$$
\mathbb{E}[\text{Tr}(M^{2r})] = \frac{1}{n^{r}}\sum_{i_{1},i_{2},\dots,i_{2r}=1}^{n}\mathbb{E}[x_{i_{1}i_{2}}x_{i_{2}i_{3}} \dots x_{i_{(2r-1)}i_{2r}}x_{i_{2r}i_{1}}].
$$

    For the intra chain merging of the chain $i_{1} \rightarrow i_{2} \rightarrow \dots \rightarrow i_{2r} \rightarrow i_{1}$, the only possible $k$-tuples with asymptotically non zero contributions are
$$
i_{1}\rightarrow i_{2}\rightarrow \dots \rightarrow i_{r} \rightarrow  i_{1}\rightarrow i_{2}\rightarrow \dots \rightarrow i_{r} \rightarrow i_{1}
$$
and
$$
i_{1}\rightarrow i_{2}\rightarrow \dots \rightarrow i_{r} \rightarrow n+1-i_{1}\rightarrow n+1-i_{2} \rightarrow  \dots n+1-i_{r}\rightarrow i_{1}
$$
All other, except these two, give zero asymptotic contribution because of reduction in degree of freedom, see section \ref{sec:Real_Single_Chain_Expectation}.
Thus,
\begin{align}
   \mathbb{E}[\Tr(M^k)]\notag
=\left\{\begin{array}{ll}
 2  & \text{ if } k= 2,4,\dots \\
0  & \text{ if } k= 1,3,\dots
\end{array}\right. + o(1).
\end{align}
\end{proof}

\subsection{Double Chain Expectation}\label{sec:Real_Double_Chain_Expectation}
Here, we consider the expectations of the form

$$
\frac{1}{n^{(k+l)/2}} \sum_{{i_{1},i_{2},\dots i_{k} , j_{1}, j_{2}, \dots j_{l}=1}}^{n} \mathbb{E}\left[(x_{i_{1}i_{2}}x_{i_{2}i_{3}}\dots x_{i_{k}i_{1}})( x_{j_{1}j_{2}} x_{j_{2}j_{3}}\dots  x_{j_{l}j_{1}})\right].
$$
Here $k, l\in \mathbb{N}$ are the lengths of the respective chains. The random variables from the first chain $x_{i_{1}i_{2}}x_{i_{2}i_{3}}\dots x_{i_{k}i_{1}}$ can be merged with the random variables from the second chain $ x_{j_{1}j_{2}} x_{j_{2}j_{3}}\dots  x_{j_{l}j_{1}}$ in three different ways; intra chain merging, cross chain merging, and partial chain merging. However, notice that all three possible pairings depend on the lengths of the chains. This is elaborated further in the Lemmas \ref{lemma:not equal power_real} and \ref{lemma:equal power_real}.
\begin{lemma}{\label{lemma:not equal power_real}} Let $M$ be same as in Theorem \ref{thm: main theorem}. Then we have
    \begin{align}
       & \mathbb{E}\left[\Tr(M^k)\Tr(M^l)\right]\notag\\
        &=\left\{\begin{array}{ll}
 0 & \text{if } k \neq l \text{ and both are odd }\\
4 & \text{ if } k \neq l \text{ and both are even }
\end{array}\right. + o(1) .\notag
    \end{align}
\end{lemma}
\begin{proof} The lemma estimates the limiting expectation of double chain in two different cases; $k\neq l$ with both odd and $k\neq l$ with both even. We present the proofs in two different parts.

\textbf{Case 1 ($k\neq l$ and both are odd):} Without loss of generality, let us assume $k<l$ and suppose $k=2p+1, \; l=2q+1$. Consider,
\begin{align*}
 &\mathbb{E}\left[ \Tr(M^{2p+1})\Tr(M^{2q+1})\right]\\
 =&\frac{1}{n^{p+q+1}}\displaystyle \sum_{{i_{1},i_{2},\dots i_{(2p+1),} \atop j_{1}, j_{2}, \dots j_{(2q+1)}=1}}^{n}\mathbb{E}\left[(x_{i_{1}i_{2}}\dots x_{i_{(2p+1)}i_{1}})(x_{j_{1}j_{2}}\dots x_{j_{(2p+1)}j_{(2p+2)}}x_{j_{(2p+2)}j_{(2p+3)}}\dots x_{j_{(2q+1)}j_{1}})\right]\\
      \end{align*}
So, the possible pairings for the double chain are as follows.
\begin{enumerate}
\item \textbf{Intra chain merging:} By the case 1 of Lemma \ref{lemma: one_chain_real}, we conclude that the intra chain merging gives zero contribution.
\item \textbf{Cross chain merging:} Since the chain lengths are unequal, cross chain merging is not possible.
\item \textbf{Partial chain merging:} To obtain the maximum contribution from partial chain merging, we divide the second chain into two parts: one with the length of the first chain and one with the remaining elements, as follows
$$(x_{i_{1}i_{2}}x_{i_{2}i_{3}}\dots x_{i_{(2p+1)}i_{1}}) (x_{j_{1}j_{2}}x_{j_{2}j_{3}}\dots x_{j_{(2p+1)}j_{(2p+2)}})(x_{j_{(2p+2)}j_{2p+3}}x_{j_{(2p+3)}j_{2p+4}}\dots x_{j_{(2q+1)}j_{1}}).$$
\end{enumerate}
Now, the first and second chains have sequential cross merging, while the third chain has intra chain merging. All other types of partial mergings have fewer degrees of freedom and therefore contribute less than this scenario. 

By the same argument as in \cite[Case 2(b), Section 7.3]{jana2024spectrum}, it can be shown that the contribution from the sequential cross chain merging of the first and second chains is $2(2p+1)n^{(2p+1)}.$

Now, for intra chain merging of the third chain, using the argument in Section \ref{sec:Real_Single_Chain_Expectation} (Case 2 Non-sequential intra chain merging), we obtain a contribution of 
$n^{q-p-1}$ as follows. The third chain is $j_{2p+2}\rightarrow j_{2p+3}\rightarrow \dots \rightarrow j_{2q+1}\rightarrow j_{1}$. For sequential intra chain merging of this chain, we have the following possibilities.
$$j_{2p+2}\rightarrow j_{2p+3}\rightarrow \dots \rightarrow j_{p+q+1}\rightarrow j_{2p+2}\rightarrow j_{2p+3}\rightarrow \dots \rightarrow j_{p+q+1}$$
and 
$$
j_{2p+2}\rightarrow j_{2p+3}\rightarrow \dots \rightarrow j_{p+q+1}\rightarrow n+1-j_{2p+2}\rightarrow n+1-j_{2p+3}\rightarrow \dots \rightarrow n+1-j_{p+q+1}.
$$
Each of the above pairings has $(q-p-1)$ free variables. Note that when we merged the first and second chains, the variable $j_{2p+2}$ has already been fixed. Therefore, the intra chain merging of the third chain contributes $2n^{q-p-1}$.

 Additionally, we chose the second and third chains from a single chain, and this can be done in $2q+1$ many ways because we can start the second chain in $2q+1$ many way.

Therefore, in total, we have a contribution of
$$(2(2p+1)n^{(2p+1)})\times(2n^{(q-p-1})\times(2q+1)+o(n^{p+q}).$$ 
However, even in this case, $\mathbb{E}\left[ \Tr(M^{2p+1})\Tr(M^{2q+1})\right]$ asymptotically tends to zero because its expression is divided by $n^{(p+q+1)/2}.$
Thus,
\begin{align*}
&\mathbb{E}\left[ \Tr(M^{2p+1})\Tr(M^{2q+1})\right] \\
&= \frac{1}{n^{p+q+1}} \sum_{{i_{1},i_{2},\dots i_{(2p+1)} \atop j_{1}, j_{2}, \dots j_{(2q+1)}=1}}^{n} \mathbb{E}\left[(x_{i_{1}i_{2}}x_{i_{2}i_{3}}\dots x_{i_{(2p+1)}i_{1}})(x_{j_{1}j_{2}}x_{j_{2}j_{3}}\dots x_{j_{(2p+1)}j_{(2p+2)}}\dots x_{j_{(2q+1)}j_{1}})\right] \\
&= \frac{1}{n^{p+q+1}} \sum_{{i_{1},i_{2},\dots i_{(2p+1)} \atop j_{1}, j_{2}, \dots j_{(2q+1)}=1}}^{n} \mathbb{E}\left[(x_{i_{1}i_{2}}x_{i_{2}i_{3}}\dots x_{i_{(2p+1)}i_{1}}) (x_{j_{1}j_{2}}x_{j_{2}j_{3}}\dots x_{j_{(2p+1)}j_{(2p+2)}})\right.\\
&\quad\quad\quad\quad\quad\quad\quad\quad\quad\quad\quad \left. (x_{j_{(2p+2)}j_{2p+3}}x_{j_{(2p+3)}j_{2p+4}}\dots x_{j_{(2q+1)}j_{2}})\right] \\
&= \frac{(2(2p+1)n^{2p+1})(2n^{(q-p-1)})(2q+1)+o(n^{p+q+1})}{n^{p+q+1}} \\
&= \frac{4(2p+1)(2q+1)n^{(p+q)}+o(n^{p+q+1})}{n^{p+q+1}}.
\end{align*}
We now explain the other case.

\textbf{Case 2 ($k \neq l$ and both are even):}
Without loss of generality, let us assume $k < l$ and write $k = 2r$, $l = 2s$ with $r, s \in \mathbb{N}$. Consider
$$\mathbb{E}\left[\Tr(M^{2r})\Tr(M^{2s})\right]
=\frac{1}{n^{r+s}}\displaystyle \sum_{{i_{1},i_{2},\dots i_{2r}, \atop j_{1}, j_{2}, \dots j_{2s}=1}}^{n}\mathbb{E}\left[(x_{i_{1}i_{2}}x_{i_{2}i_{3}}\dots x_{i_{2r}i_{1}})(x_{j_{1}j_{2}}x_{j_{2}j_{3}}\dots x_{j_{2s}j_{1}})\right]$$
Here, the chain lengths are even, so non zero contribution from intra chain merging is possible. From sequential intra chain merging of the first chain, we obtain a contribution of $2n^r$. Similarly, sequential intra chain merging of the second chain gives a contribution of $2n^s$. Therefore, in total, we have a contribution of $4n^{r+s}$. Partial merging, as explained in the previous case, has $o(n^{r+s})$ degrees of freedom and does not contribute asymptotically.
Thus,
$$ \mathbb{E}\left[\Tr(M^{2r})\Tr(M^{2s})\right] = \frac{4n^{r+s}+o(n^{r+s})}{n^{r+s}}.$$

    \end{proof}

\begin{rem}\label{Remark:when one of the powers is odd and other is even_real}
    While computing $\mathbb{E}\left[ \Tr(M^{k})\Tr(M^{l})\right]$, there is another scenario to consider, in addition to Lemma \ref{lemma:not equal power_real}, namely if one of $k$ and $l$ is even, and the other is odd. In this scenario, intra chain merging gives zero asymptotic contribution due to the odd length of one chain. Moreover, partial merging also results in zero asymptotic contribution. Consequently, this case contributes zero asymptotically.
\end{rem}

\begin{lemma}{\label{lemma:equal power_real}}
Let $M$ be same as in Theorem \ref{thm: main theorem}. Then we have 
    \begin{align}
       &\frac{1}{n^k}\mathbb{E}\left[\Tr(M^k)\Tr(M^k)\right]\notag
       =\left\{\begin{array}{ll}
       2k  & \text{if } k \text{ is odd}\\
       2k+4 & \text{ if } k \text{ is even}
   \end{array}\right.  + o(1).\notag
    \end{align}
    \end{lemma}  
\begin{proof}
The analysis is slightly different depending on whether $k$ is even or odd. We discuss both cases separately as follows.

\textbf{Case 1 ($k$ is even):} 
Let $k=2u,\; u\in \mathbb{N}.$ 
The expectation of the square of the trace of $M^{2u}$ can be expressed as follows.
\begin{align*}
\mathbb{E}[(\Tr(M^{2u}))^2] 
&= \frac{1}{n^{2u}} \sum_{{i_{1},i_{2},\dots i_{2u} \atop j_{1}, j_{2}, \dots j_{2u}= 1}}^{n}\mathbb{E}\left[(x_{i_{1}i_{2}}x_{i_{2}i_{3}}\dots x_{i_{2u}i_{1}})(x_{j_{1}j_{2}}x_{j_{2}j_{3}}\dots x_{j_{2u}j_{1}})\right].
\end{align*}
Here, the chain lengths are even, so by intra chain merging of both chains, we obtain a contribution of $2n^{u} \times 2n^{u}$. Since the sizes of both chains are equal, we also obtain a significant contribution from sequential cross chain merging, which contributes $2(2u n^{2u})$. All other mergings give asymptotically zero contribution. Therefore, in total, we have $$4n^{2u} + 2(2un^{2u})+o(n^{2u}) = 4(u+1)n^{2u}+o(n^{2u})$$ contribution.

\textbf{Case 2 ($k$ is odd):} 
Let $k=2v+1$, where $v\in \mathbb{Z}^+$.  This case gives the same contribution as in the previous case (Case 1 ($k$ even)), except that here intra chain merging gives zero contribution because the sizes of the chains are odd. Therefore, in this case, we obtain  
$$
2(2v+1)n^{2v+1} + o(n^{2v+1})
$$  
contribution.
\end{proof} 

We now complete the proof of Proposition \ref{Prop:Var} below.

\begin{proof}[Proof of Proposition \ref{Prop:Var}] We first find the expression of $V_{f},$ when $f(z)=\sum_{k=1}^{d}a_{k}z^{k}$ is a polynomial with real coefficients.
Using Lemmas \ref{lemma: one_chain_real}, \ref{lemma:not equal power_real}, \ref{lemma:equal power_real}, and Remark \ref{Remark:when one of the powers is odd and other is even_real} in \eqref{eq:Var_1}, we write the variance as 
\begin{align} \label{eq:Var_1}
&V_{f} \notag \\
&= \Var\left(\sum_{k=2}^{d}a_{k} \Tr(M^{k})\right) \notag \\
&= \sum_{k,l=1}^{d}  a_{k} \bar{a_{l}} \Cov\left(\Tr(M^k),\Tr(M^l)\right) \notag \\
&= \sum_{k,l=1}^{d} a_{k} \bar{a_{l}}\left(\mathbb{E}\left[\Tr(M^k)\Tr(M^l)\right]-\mathbb{E}\left[\Tr(M^k)\right] \mathbb{E}\left[\Tr(M^l)\right]\right) \notag\\
&=\sum_{\substack{k\neq l \\ k,l=1,5,\dots}} 0 + \sum_{\substack{k\neq l \\ k,l=2,4,\dots}}4a_{k} \bar{a_{l}}+\sum_{\substack{k=l \\ k,l=1,5,\dots}}2k a_k \bar{a_l}\\
&\;\;\;+ \sum_{\substack{k=l \\ k,l=2,4,\dots}}2(k+2)a_{k} \bar{a_{l}}-\sum_{k,l=2,4,\dots}4a_{k} \bar{a_{l}}+o(1)\notag\\
&= \sum_{k=1}^d 2k|a_k|^2 + o(1).\notag
\end{align}

We now proceed to calculate the limiting variance $V_{f}$ for analytic test function. Let $f$ be an analytic function on $\mathbb{D}$. The eigenvalue distribution of centrosymmetric random matrices converge to the circular law, hence all eigenvalues $\{\lambda_i(M)\}_{i=1}^n$ lie inside $\mathbb{D}$ with probability $1$ as $n \to \infty$. We can write the LES as follows.
\begin{align}
\operatorname{L}_{n}(f)\notag
&= \sum_{i=1}^{n} f(\lambda_{i}) \notag \\
&= (2\pi i)^{-1} \oint_{\partial \mathbb{D}} \sum_{i=1}^{n} \frac{f(z)}{z-\lambda_{i}}\, dz \notag \\
&= (2\pi i)^{-1} \oint_{\partial \mathbb{D}} f(z) \Tr \mathcal{R}_{z}(M) \, dz, \notag
\end{align}
where in the second equality we have applied Cauchy’s integral formula for analytic functions, and $\mathcal{R}_z(M)= (zI-M)^{-1}$ denotes the resolvent of the matrix $M$.  

Now, the centered LES can be expressed as follows.
$$
\operatorname{L}_{n}^{\circ}(f) = \frac{1}{2\pi i} \oint_{\partial \mathbb{D}} f(z) \Tr {\mathcal{R}}_{z}^{\circ}(M) \, dz.
$$
where $\psi ^{\circ} = \psi -\mathbb{E}[\psi ].$ To compute the variance of $\operatorname{L}_{n}^{\circ}(f)$, we require the correlation of centered resolvent traces
$\mathbb{E}\!\left[\Tr \mathcal{R}_z^{\circ}(M)\, \Tr \mathcal{R}_{\bar\eta}^{\circ}( M)\right].$ We expand $\Tr{\mathcal{R}_z(M)}$ on the boundary $\partial{\mathbb{D}}$ as follows.
$$
\Tr{\mathcal{R}_z(M)} = \sum_{k=0}^\infty z^{-k-1}\Tr M^k,
$$
and similarly for $\Tr \mathcal{R}_{\bar\eta}(M)$. 
Hence by using the Lemmas \ref{lemma:not equal power_real} and \ref{lemma:equal power_real}, we have    
\begin{align*}
    &\mathbb{E}[\Tr \mathcal{R}_{z}^{\circ}(M)\Tr\mathcal{R}_{\bar\eta}^{\circ}(M)]\\
    &=\left\{\mathbb{E}[\Tr \mathcal{R}_{z}(M)\Tr\mathcal{R}_{\bar\eta}(M)]-\frac{n^{2}}{z\bar \eta}\right\}\\
    &=\sum_{k=1}^{\infty}(z\bar \eta)^{-k-1}\mathbb{E}[\Tr M^{k}\Tr M^{k}] + \sum_{\stackrel{k,l=1}{k\neq l}}^{\infty}z^{-k-1}(\bar \eta)^{-l-1}\mathbb{E}[\Tr{M^{k}}\Tr{M^{l}}]\\
    &=2\sum_{k=1}^{\infty}k(z\bar \eta)^{-k-1}+4\sum_{\stackrel{k=2}{k: \text{ even}}}^{\infty}(z\bar \eta)^{-k-1}+4\sum_{\stackrel{k,l=1}{k\neq l}}^{\infty}z^{-k-1}(\bar \eta)^{-l-1}\\
    &=2\left(1-z\bar\eta\right)^{-2}+\frac{4}{z\bar\eta((z\bar\eta)^{2}-1)}+4\left( \frac{1}{z\bar\eta\,(z-1)(\bar\eta-1)} \;-\; \frac{1}{z\bar\eta\,(z\bar\eta-1)} \right).
\end{align*}
Thus, the final expression of the limiting variance becomes
\begin{align*}
&V_f\\
=&\mathbb{E}\left[\big|\operatorname{L}_{n}^{\circ}(f)\big|^2\right]\\ 
=& -\frac{1}{4\pi^2} 
\oint_{\partial \mathbb{D}} \oint_{\partial \mathbb{D}} 
f(z)\, f(\bar{\eta}) \,
\mathbb{E}\!\left[\Tr \mathcal{R}_z^{\circ}(M)\, \Tr \mathcal{R}_{\bar{\eta}}^{\circ}(M)\right] 
\, dz\, d\bar{\eta} \\
=& -\frac{1}{4\pi^2} 
\oint_{\partial \mathbb{D}} \oint_{\partial \mathbb{D}} 
f(z)\, f(\bar{\eta}) \Bigg(
2\left(1-z\bar\eta\right)^{-2}+\frac{4}{z\bar\eta((z\bar\eta)^{2}-1)}+4\left( \frac{1}{z\bar\eta\,(z-1)(\bar\eta-1)} \;-\; \frac{1}{z\bar\eta\,(z\bar\eta-1)} \right)\Bigg) 
\, dz\, d\bar{\eta} \\
\end{align*}
\end{proof}

\subsection*{Acknowledgment} Indrajit Jana's research is partially supported by INSPIRE Fellowship\\DST/INSPIRE/04/2019/000015, Dept. of Science and Technology, Govt. of India.\\

Sunita Rani's research is fully supported by the University Grant Commission (UGC), New Delhi.

\bibliographystyle{abbrv} 

\begin{thebibliography}{10}
	
	\bibitem{adhikari2017fluctuations}
	K.~Adhikari and K.~Saha.
	\newblock Fluctuations of eigenvalues of patterned random matrices.
	\newblock {\em Journal of Mathematical Physics}, 58(6):063301, 2017.
	
	\bibitem{adhikari2018universality}
	K.~Adhikari and K.~Saha.
	\newblock Universality in the fluctuation of eigenvalues of random circulant
	matrices.
	\newblock {\em Statistics \& Probability Letters}, 138:1--8, 2018.
	
	\bibitem{anderson2006clt}
	G.~W. Anderson and O.~Zeitouni.
	\newblock A clt for a band matrix model.
	\newblock {\em Probability Theory and Related Fields}, 134(2):283--338, 2006.
	
	\bibitem{bai2008clt}
	Z.~D. Bai and J.~W. Silverstein.
	\newblock Clt for linear spectral statistics of large-dimensional sample
	covariance matrices.
	\newblock In {\em Advances In Statistics}, pages 281--333. World Scientific,
	2008.
	
	\bibitem{BardenetHardy2020}
	R.~Bardenet and A.~Hardy.
	\newblock Monte carlo with determinantal point processes.
	\newblock {\em The Annals of Applied Probability}, 30(1):368--417, 2020.
	
	\bibitem{BourgadeErdosYau2014}
	P.~Bourgade, L.~Erd{\H{o}}s, and H.-T. Yau.
	\newblock Edge universality of beta ensembles.
	\newblock {\em Communications in Mathematical Physics}, 332(1):261--353, 2014.
	
	\bibitem{chatterjee2009fluctuations}
	S.~Chatterjee.
	\newblock Fluctuations of eigenvalues and second order poincar{\'e}
	inequalities.
	\newblock {\em Probability Theory and Related Fields}, 143(1):1--40, 2009.
	
	\bibitem{cipolloni2022fluctuations}
	G.~Cipolloni.
	\newblock Fluctuations in the spectrum of non-hermitian iid matrices.
	\newblock {\em Journal of Mathematical Physics}, 63(5), 2022.
	
	\bibitem{cipolloni2023central}
	G.~Cipolloni, L.~Erd{\H{o}}s, and D.~Schr{\"o}der.
	\newblock Central limit theorem for linear eigenvalue statistics of
	non-hermitian random matrices.
	\newblock {\em Communications on Pure and Applied Mathematics},
	76(5):946--1034, 2023.
	
	\bibitem{cipolloni2021fluctuation}
	G.~Cipolloni, L.~Erdős, and D.~Schr{\"o}der.
	\newblock {Fluctuation around the circular law for random matrices with real
		entries}.
	\newblock {\em Electronic Journal of Probability}, 26(none):1 -- 61, 2021.
	
	\bibitem{jana2022clt}
	I.~Jana.
	\newblock Clt for non-hermitian random band matrices with variance profiles.
	\newblock {\em Journal of Statistical Physics}, 187(2):13, 2022.
	
	\bibitem{jana2025cltles}
	I.~Jana and S.~Rani.
	\newblock Clt for les of correlated non-hermitian random matrices, 2025.
	\newblock arXiv preprint arXiv:2503.22542.
	
	\bibitem{jana2024spectrum}
	I.~Jana and S.~Rani.
	\newblock Spectrum of random centrosymmetric matrices; clt and circular law.
	\newblock {\em Random Matrices: Theory and Applications}, 14(01):2450026, 2025.
	
	\bibitem{jana2016fluctuations}
	I.~Jana, K.~Saha, and A.~Soshnikov.
	\newblock Fluctuations of linear eigenvalue statistics of random band matrices.
	\newblock {\em Theory of Probability \& Its Applications}, 60(3):407--443,
	2016.
	
	\bibitem{johansson1998fluctuations}
	K.~Johansson.
	\newblock {On fluctuations of eigenvalues of random Hermitian matrices}.
	\newblock {\em Duke Mathematical Journal}, 91(1):151 -- 204, 1998.
	
	\bibitem{Jonsson1982SomeLT}
	D.~Jonsson.
	\newblock Some limit theorems for the eigenvalues of a sample covariance
	matrix.
	\newblock {\em Journal of Multivariate Analysis}, 12:1--38, 1982.
	
	\bibitem{kopel2015linear}
	P.~Kopel.
	\newblock Linear statistics of non-hermitian matrices matching the real or
	complex ginibre ensemble to four moments.
	\newblock {\em arXiv preprint arXiv:1510.02987}, 2015.
	
	\bibitem{li2013central}
	L.~Li and A.~Soshnikov.
	\newblock Central limit theorem for linear statistics of eigenvalues of band
	random matrices.
	\newblock {\em Random Matrices: Theory and Applications}, 02(04):1350009, 2013.
	
	\bibitem{liu2012fluctuations}
	D.~Liu, X.~Sun, and Z.~Wang.
	\newblock Fluctuations of eigenvalues for random toeplitz and related matrices.
	\newblock {\em Electronic Journal of Probability}, 17:1--22, 2012.
	
	\bibitem{lytova2009central}
	A.~Lytova and L.~Pastur.
	\newblock {Central limit theorem for linear eigenvalue statistics of random
		matrices with independent entries}.
	\newblock {\em The Annals of Probability}, 37(5):1778 -- 1840, 2009.
	
	\bibitem{maurya2021fluctuations}
	S.~N. Maurya and K.~Saha.
	\newblock Fluctuations of linear eigenvalue statistics of reverse circulant and
	symmetric circulant matrices with independent entries.
	\newblock {\em Journal of Mathematical Physics}, 62(4):043506, 2021.
	
	\bibitem{4493413}
	R.~R. Nadakuditi and A.~Edelman.
	\newblock Sample eigenvalue based detection of high-dimensional signals in
	white noise using relatively few samples.
	\newblock {\em IEEE Transactions on Signal Processing}, 56(7):2625--2638, 2008.
	
	\bibitem{o2016central}
	S.~O’Rourke and D.~Renfrew.
	\newblock Central limit theorem for linear eigenvalue statistics of elliptic
	random matrices.
	\newblock {\em Journal of Theoretical Probability}, 29:1121--1191, 2016.
	
	\bibitem{rider2006gaussian}
	B.~Rider and J.~W. Silverstein.
	\newblock {Gaussian fluctuations for non-Hermitian random matrix ensembles}.
	\newblock {\em The Annals of Probability}, 34(6):2118 -- 2143, 2006.
	
	\bibitem{serfaty2024lecturescoulombrieszgases}
	S.~Serfaty.
	\newblock Lectures on coulomb and riesz gases, 2024.
	
	\bibitem{shcherbina2015fluctuations}
	M.~Shcherbina.
	\newblock On fluctuations of eigenvalues of random band matrices.
	\newblock {\em Journal of Statistical Physics}, 161(1):73--90, 2015.
	
	\bibitem{sinai1998central}
	Y.~Sinai and A.~Soshnikov.
	\newblock Central limit theorem for traces of large random symmetric matrices
	with independent matrix elements.
	\newblock {\em Boletim da Sociedade Brasileira de
		Matem{\'a}tica-Bulletin/Brazilian Mathematical Society}, 29(1):1--24, 1998.
	
	\bibitem{weaver1985centrosymmetric}
	J.~R. Weaver.
	\newblock Centrosymmetric (cross-symmetric) matrices, their basic properties,
	eigenvalues, and eigenvectors.
	\newblock {\em The American Mathematical Monthly}, 92(10):711--717, 1985.
	
\end{thebibliography}

\end{document}